\documentclass[twoside, 12pt]{article}
\usepackage{amssymb, amsmath, mathrsfs, amsthm}
\usepackage{graphicx}
\usepackage{color,xcolor}
\usepackage{float}

\newcommand{\bd}{\begin{description}}
\newcommand{\ed}{\end{description}}
\newcommand{\bi}{\begin{itemize}}
\newcommand{\ei}{\end{itemize}}
\newcommand{\bt}{\begin{tabular}}
\newcommand{\et}{\end{tabular}}
\newcommand{\beq}{\begin{equation}}
\newcommand{\eeq}{\end{equation}}
\newcommand{\be}{\begin{enumerate}}
\newcommand{\ee}{\end{enumerate}}
\newcommand{\beqs}{\begin{eqnarray*}}
\newcommand{\eeqs}{\end{eqnarray*}}

\definecolor{DarkGreen}{rgb}{0.2, 0.6, 0.3}

\newtheorem{theorem}{Theorem}
\newtheorem{lemma}{Lemma}

\newtheorem{corollary}[theorem]{Corollary}
\newtheorem{case}{Case}

\newtheorem{claim}{Claim}

\newtheorem{problem}{Problem}
\newtheorem{proposition}{Proposition}

\newtheorem{conjecture}{Conjecture}

\newcommand{\labelz}[1]{\label{#1}}

\newcommand{\ch}{\mathscr{C}}

\begin{document}
\title{\bf Gallai-Ramsey numbers for rainbow paths\footnote{Research partially supported by National Natural Science Foundation of China (No. 11871398)}}
\author{Xihe Li\footnote{Department of Applied Mathematics, School of Science, Northwestern Polytechnical University, Xi'an, Shaanxi 710072, P.~R.~China.} \footnote{Xi'an-Budapest Joint Research Center for Combinatorics, Northwestern Polytechnical University, Xi'an, Shaanxi 710129, P.~R.~China.}, Pierre Besse\footnote{Department of Mathematics, Clayton State University, Morrow, GA 30260, USA.}, Colton Magnant\footnotemark[4],\\ Ligong Wang\footnotemark[2] \footnotemark[3], Noah Watts\footnotemark[4]}
\maketitle

\begin{abstract}
Given graphs $G$ and $H$ and a positive integer $k$, the \emph{Gallai-Ramsey number}, denoted by $gr_{k}(G : H)$ is defined to be the minimum integer $n$ such that every coloring of $K_{n}$ using at most $k$ colors will contain either a rainbow copy of $G$ or a monochromatic copy of $H$. We consider this question in the cases where $G \in \{P_{4}, P_{5}\}$. In the case where $G = P_{4}$, we completely solve the Gallai-Ramsey question by reducing to the $2$-color Ramsey numbers. In the case where $G = P_{5}$, we conjecture that the problem reduces to the $3$-color Ramsey numbers and provide several results in support of this conjecture.
\end{abstract}

\section{Introduction}

In this work, we only consider edge colorings of graphs. A colored graph is called \emph{rainbow} if all edges have different colors and \emph{monochromatic} if all edges have a single color. Given a graph $G$, the $k$-color Ramsey number for $G$, denoted by $R_{k}(G)$, is the minimum integer $n$ such that every coloring of $K_{n}$ using at most $k$ colors will contain a monochromatic copy of $G$ in some color. Given graphs $G$ and $H$ and a positive integer $k$, the \emph{Gallai-Ramsey number}, denoted by $gr_{k}(G : H)$ is defined to be the minimum integer $n$ such that every coloring of $K_{n}$ using at most $k$ colors will contain either a rainbow copy of $G$ or a monochromatic copy of $H$. Other standard notation can be found in \cite{CLZ11}.

Recently, there have been many results concerning the case where $G$ is a triangle. We refer the interested reader to the survey \cite{MR2606615} with a dynamically updated version available at \cite{FMO14}. Other choices for $G$ have been much less studied so we consider the case where $G$ is a path. For short paths, the structure of colored complete graphs containing no rainbow path is well understood (see Theorems~\ref{Thm:P4Classify} and~\ref{Thm:P5Classify}).

When $G = P_{4}$, the structure is extremely strong, yielding the following result.

\begin{theorem}\labelz{Thm:P4Reduction}
For any graph $H$ with no isolated vertices, we have
$$
gr_{k}(P_{4} : H) = R_{2}(H)
$$
except when $H = P_{3}$ and $k \geq 3$, in which case
$$
gr_{k}(P_{4} : P_{3}) = 5.
$$
\end{theorem}

Note that the restriction of isolated vertices is simply to eliminate trivial case analysis and can be avoided by ensuring there are enough remaining vertices.

Theorem~\ref{Thm:P4Reduction} actually completes the classification that was begun in \cite{MR3539092} with the following result.

\begin{theorem}[\cite{MR3539092}]
For every graph $H$ of order $n \geq 5$, $gr_{k}(P_{4} : H) = R_{2}(H)$.
\end{theorem}

When $G = P_{5}$, the structure is not quite as strong as in the $P_{4}$ case but we believe the following to be true.

\begin{conjecture}\labelz{Conj:P5Reduction}
For any graph $H$ with no isolated vertices, we have
$$
gr_{k}(P_{5} : H) = R_{3}(H).
$$
\end{conjecture}

As seen in Theorem~\ref{Thm:P4Reduction}, there may be one or more exceptional graphs but since our partial results in support of this conjecture eliminate many of the most natural candidates, we feel this conjecture is likely to be true in its stated form.

For the lower bound, the sharpness example for $R_{3}(H)$ is a $3$-colored complete graph containing no monochromatic copy of $H$. This trivially contains no rainbow copy of $P_{5}$ since only three colors are used. It therefore suffices to prove (or disprove) the upper bound in Conjecture~\ref{Conj:P5Reduction}.

This paper is laid out as follows. In Section~\ref{Sec:Prelim}, we review several preliminary results that will be used later in the proofs. These include the aforementioned structural characterizations of graphs with no rainbow small paths. In Section~\ref{Sec:P4}, we prove Theorem~\ref{Thm:P4Reduction}. Finally, Section~\ref{Sec:P5} contains several results in support of Conjecture~\ref{Conj:P5Reduction}.

\section{Preliminaries}\labelz{Sec:Prelim}

We first state the main structural tools that will be used in out proofs. These provide strong structure when short rainbow paths are forbidden.

\begin{theorem}[\cite{TW07}]\labelz{Thm:P4Classify}
Let $K_n, n\geq 4$, be edge colored such that it contains no rainbow $P_4$. Then one of the following holds:
\bd
\item{\rm (a)} at most two colors are used;
\item{\rm (b)} $n = 4$ and the graph is precisely the $3$-coloring of $K_{4}$ in which each color induces a matching.
\ed
\end{theorem}

For this next result, given a color $j$, let $V^{(j)}$ be the set of vertices with at least one incident edge in color $j$ and let $E^{(j)}$ be the set of edges of color $j$.

\begin{theorem}[\cite{TW07}]\labelz{Thm:P5Classify}
Let $K_n, n\geq 5$, be edge colored such that it contains no rainbow $P_5$. Then after renumbering the colors, one of the following holds:
\bd
\item{\rm (a)} at most three colors are used;
\item{\rm (b)} color $1$ is dominant, meaning that the sets $V^{(j)}, j\geq 2$, are disjoint;
\item{\rm (c)} $K_n - v$ is monochromatic for some vertex $v$;
\item{\rm (d)} there exist three special vertices $v_1, v_2, v_3$ such that $E^{(2)}=\{v_1v_2\}, E^{(3)}=\{v_1v_3\}, E^{(4)}$ contains $v_2v_3$ plus perhaps some edges incident with $v_1$, and every other edge is in $E^{(1)}$;
\item{\rm (e)} there exist four special vertices $v_1, v_2, v_3, v_4$ such that $\{v_1v_2\} \subset E^{(2)} \subset \{v_1v_2, v_3v_4\}$, $E^{(3)}=\{v_1v_3, v_2v_4\}$, $E^{(4)}=\{v_1v_4, v_2v_3\}$, and every other edge is in $E^{(1)}$;
\item{\rm (f)} $n=5$, $V(K_n)=\{v_1, v_2, v_3, v_4, v_5\}$, $E^{(1)}=\{v_1v_4, v_1v_5, v_2v_3\}$, $E^{(2)}=\{v_2v_4, v_2v_5, v_1v_3\}$, $E^{(3)}=\{v_3v_4, v_3v_5, v_1v_2\}$ and $E^{(4)}=\{v_4v_5\}$.
\ed 
\end{theorem}

More generally, let $\mathscr{G}$ be a non-empty set of graphs. Let $R_2(\mathscr{G})$ be the minimum number of vertices $n$ such that in every $2$-coloring of $K_{n}$, there is a monochromatic copy of some graph in $\mathscr{G}$. More specifically, for two sets of graphs $\mathscr{G}_{1}$ and $\mathscr{G}_{2}$, let $R(\mathscr{G}_{1}, \mathscr{G}_{2})$ be the minimum number of vertices $n$ such that in every red-blue coloring of $K_{n}$, there is either a red copy of a graph in $\mathscr{G}_{1}$ or a blue copy of a graph in $\mathscr{G}_{2}$. If either set consists of a single graph, the notation will be simplified to just be the graph, e.g.~$R(\mathscr{G}, G)$.

Given a bipartite graph $H = A \cup B$ say with $|A| \geq |B|$, let $b(H) = |A|$ denote the order of the bigger side of $H$ and let $s(H) = |B|$ denote the smaller side of $H$.

\section{Proof of Theorem~\ref{Thm:P4Reduction}}\labelz{Sec:P4}

In this section, we provide the straightforward proof of Theorem~\ref{Thm:P4Reduction}. As simple as this proof is, it provides an introduction to some of the strategies that will be used in our later results.

\begin{proof}
For the lower bound, the sharpness example for $R_{2}(H)$ is a $2$-colored complete graph on $R_{2}(H) - 1$ vertices containing no monochromatic copy of $H$. This trivially also contains no rainbow copy of $P_{4}$ since only two colors are used. In the special case when $H = P_{3}$, we have $R_{2}(P_{3}) = 3$ but the graph described in Case~(b) of Theorem~\ref{Thm:P4Classify}, the $3$-coloring of $K_{4}$ in which each color induces a matching, contains no rainbow copy of $P_{4}$ and no monochromatic copy of $P_{3}$.

For the upper bound, we consider a coloring $G$ of $K_{n}$ where $n = R_{2}(H)$ which contains no rainbow copy of $P_{4}$. By Theorem~\ref{Thm:P4Classify}, there are only two possible cases for what this coloring can look like. If Case~(a) holds, then $G$ uses only two colors and there is a monochromatic copy of $H$ in $G$ by the definition of $R_{2}(H)$. On the other hand, if Case~(b) holds, then $n = 4$, which is a contradiction unless $H \in \{P_{2}, P_{3}\}$ since $R_{2}(H) > 4$ for any other graph $H$.

If $H = P_{2}$, then trivially $gr_{k}(P_{4} : P_{2}) = R_{2}(P_{2}) = 2$. If $H = P_{3}$, then the $3$-coloring of $K_{4}$ in which each color induces a matching contains no rainbow $P_{4}$ and no monochromatic copy of $P_{3}$. By Theorem~\ref{Thm:P4Classify}, for $k \geq 3$, there is no $k$-coloring of $K_{n}$ with $n \geq 5$ which does not contain a rainbow $P_{4}$. This means that for $k \geq 3$, we have $gr_{k}(P_{4} : P_{3}) = 5$.
\end{proof}

\section{Rainbow $P_{5}$}\labelz{Sec:P5}

In this section, we prove several results in support of Conjecture~\ref{Conj:P5Reduction}. For the sake of notation, let $\mathscr{H}$ be the set of graphs $H$ for which $gr_{k}(P_{5} : H) > R_{3}(H)$, those that do not satisfy Conjecture~\ref{Conj:P5Reduction}. Indeed, Conjecture~\ref{Conj:P5Reduction} claims that the set $\mathscr{H}$ is empty.

\begin{lemma}\labelz{Lem:P5Disconnected+b}
Every graph $H \in \mathscr{H}$ is disconnected. Furthermore, in order for a $k$-colored complete graph containing no rainbow $P_{5}$ and no monochromatic copy of $H$ to have more than $R_{3}(H)$ vertices, it must satisfy Case~(b) of Theorem~\ref{Thm:P5Classify}.
\end{lemma}

Essentially, this lemma states that in order to prove Conjecture~\ref{Conj:P5Reduction}, it suffices to consider only disconnected graphs $H$ and look within colored complete graphs satisfying Case~(b) of Theorem~\ref{Thm:P5Classify}.

\begin{proof}
Let $H$ be a graph and $k$ be a positive integer. Let $G$ be a $k$-colored complete graph. The goal of this proof is to show that $gr_{k}(P_{5} : H) = R_{3}(H)$ if either $H$ is connected or $G$ satisfies any case of Theorem~\ref{Thm:P5Classify} other than Case~(b).

We consider a coloring $G$ of $K_n$ where $n = R_3(H)$ which contains no rainbow copy of $P_{5}$. If $G$ satisfies Case~(a), the result is immediate by the definition of $R_{3}(H)$. This means we may assume that at least $4$ colors appear in $G$. In each of the Cases~(c), (d), (e), and~(f), $G$ contains a monochromatic copy of $K_{n-2}-e$, and moreover, there is a monochromatic copy of $K_{n-3}$.
\bi
\item If $H = P_{2}$, then trivially $gr_{k}(P_{5} : P_{2}) = R_{3}(P_{2}) = 2$, so we may assume $|H| \geq 3$.
\item If $H$ is a complete graph, the $R_3(H)\geq |H|+3$ clearly, and thus there is a monochromatic copy of $H$ in $G$.
\item If $H$ is not a complete graph and not in $\{P_{3}, 2P_{2}\}$, i.e. $H$ has at least one missing edge, then $R_3(H)\geq |H|+2$, and thus there is a monochromatic copy of $H$ in $G$.
\item If $H = P_{3}$, then $n = R_{3}(P_{3}) = 5$ \cite{MR2310787}. With at least $4$ colors and no rainbow copy of $P_{5}$, $G$ must be the graph in Case~(f), which contains a monochromatic copy of $P_{3}$. 
\item If $H = 2P_{2}$, then $n = R_{3}(2P_{2}) \geq 6$. This means that $G$ must satisfy one of Cases~(c), (d), or~(e), each of which contains a monochromatic copy of $2P_{2}$.
\ei

We may therefore suppose that $G$ satisfies Case~(b) of Theorem~\ref{Thm:P5Classify}. If $H$ is connected, then merging all colors other than color $1$ into a single color would not create a monochromatic copy of $H$. Since $|G| = R_{3}(H) \geq R_{2}(H)$, there is a monochromatic copy of $H$ in $G$ to complete the proof. We may therefore assume that $H$ is disconnected.
\end{proof}

Next we prove that all bipartite graphs satisfy Conjecture~\ref{Conj:P5Reduction}.

\begin{lemma}\labelz{Lem:Bip}
If $H$ is bipartite, then $H \notin \mathscr{H}$.
\end{lemma}

\begin{proof}
Let $H$ be a bipartite graph, let $k$ be a positive integer, and suppose $G$ is a $k$-coloring of $K_{n}$ where $n = R_{3}(H)$ which contains no rainbow copy of $P_{5}$ and no monochromatic copy of $H$. By Lemma~\ref{Lem:P5Disconnected+b}, we may assume that $H$ is disconnected and $G$ satisfies Case~(b) of Theorem~\ref{Thm:P5Classify}.

Let $V_2, V_3, \ldots, V_k$ be a partition of $V(G)$ such that there are only edges of color $1$ or $i$ within $V_i$ for $2\leq i\leq k$, and there are only edges of color $1$ in between the parts. Choose a subset of colors $U\subset \{2, 3, \ldots, k\}$ and define vertex sets $A_1=\bigcup_{i\in U}V_i$ and $A_{2} = G \setminus A_{1}$ such that 
\bd
\item{(i)} $|A_1|\geq |A_2|$ and 
\item{(ii)} $|A_{1}| - |A_2|$ is minimal.
\ed

\begin{claim}\labelz{cl:1}
If $|U|\geq 2$, then $|A_2|\geq \frac{|A_1|}{2}$.
\end{claim}

\begin{proof}
Suppose, for a contradiction, that $|A_2|< \frac{|A_1|}{2}$. Since $|U|\geq 2$, we know that $\min_{i\in U}\{|V_i|\} \leq \frac{|A_1|}{2}$ where this minimum is achieved say by the part $V_{j}$, so $|V_{j}| \leq \frac{|A_{1}|}{2}$. If $|A_1\setminus V_j|\geq |A_2\cup V_j|$, then $U'=U\setminus \{j\}$ is a better choice than $U$, contradicting the choice of $U$.
We may therefore assume that $|A_1\setminus V_j| < |A_2\cup V_j|$. Now let $U'= (\{2, 3, \dots, k\} \setminus U) \cup \{j\}$, and correspondingly $A'_1=\bigcup_{i\in U'}V_i=A_2\cup V_j$ and $A'_2= G \setminus A_{2} =A_1\setminus V_j$. Then we have $|A_{1}'| \geq |A_{2}'|$ and
\beqs
(|A_1|-|A_2|)-(|A'_1|-|A'_2|) & = & |A_1|-|A_2|-((|A_2|+|V_j|)-(|A_1|-|V_j|))\\
~ & = & 2(|A_1|-|A_2|-|V_j|) \\ 
~ & > & 2\left(|A_1|-\frac{|A_1|}{2}-\frac{|A_1|}{2}\right)\\
~ & = & 0,
\eeqs
contradicting to the choice of $U$.
\end{proof}

Now we recolor the edges of $G$ to make a $3$-coloring such that

\bd
\item{(i)} change all edges of color $1$ to red;
\item{(ii)} for $i\in U$, change all edges of color $i$ to blue;
\item{(iii)} for $i\in \{2, 3, \ldots, k\}\setminus U$, change all edges of color $i$ to green.
\ed
Let $G'$ denote the resulting graph and since $|G'| = |G| = R_3(H)$, there must be a monochromatic copy of $H$, say $M \subseteq G'$. Since $G$ contains no monochromatic $H$, then $M$ must be colored by blue or green and moreover, if $|U|=1$ then $M$ must be green.

First suppose $|U|=1$ so $M$ is green. Then certainly $|A_{1}| \geq s(H)$ and $|A_{2}| \geq b(H)$. For any subsets $S\subseteq A_1$ and $B\subseteq A_2$ with $|S|=s(H)$ and $|B|=b(H)$, the vertices $S\cup B$ with corresponding edges $E(S, B)$ of color $1$ form a monochromatic copy of $K_{s(H),b(H)}$, which contains a copy of $H$, a contradiction.

Finally suppose $|U|\geq 2$. Then by Claim~\ref{cl:1}, we have $|A_i|\geq \frac{|A_{3-i}|}{2}$ for each $i \in \{1,2\}$. Without loss of generality, we may assume that $M$ is colored by blue in $G'$, which implies $|A_1|\geq |H|$.
Since $|A_2|\geq \frac{|A_1|}{2} \geq \frac{|H|}{2} \geq s(H)$, we can choose subsets of vertices $S\subseteq A_2$ with $|S|=s(H)$ and $B\subseteq A_1$ with $|B|=b(H)$. Then the vertices $S\cup B$ with corresponding edges $E(S, B)$ of color $1$ form a monochromatic copy of $K_{s(H),b(H)}$ using color $1$ in $G$, which contains a copy of $H$, a contradiction to complete the proof of Lemma~\ref{Lem:Bip}.
\end{proof}

Our next lemma may appear, on the surface, to be a relatively simple observation but it leads to a variety of other results, as presented in the subsection to follow.

\begin{lemma}\labelz{Lem:CH}
Let $H$ be a disconnected graph and $\ch(H)$ be the set of connected graphs containing $H$ as a subgraph. If $R_3(H)\geq R_2(\ch(H))$, then $gr_k(P_5 : H)=R_3(H)$.
\end{lemma}

\begin{proof}
Let $G$ be a rainbow $P_{5}$-free $k$-coloring of $K_{R_{3}(H)}$. By Lemma~\ref{Lem:P5Disconnected+b}, it suffices to consider such colorings $G$ that satisfy Case~(b) of Theorem~\ref{Thm:P5Classify}.
We recolor the edges of $G$ such that

\bd
\item{(i)} replace all edges of color $1$ with blue;
\item{(ii)} for $i\in \{2, 3, \ldots, k\}$, replace all edges of color $i$ with red.
\ed

Let $G'$ be the resulting graph and note that $|G'|=|G|=R_3(H)\geq R_2(\ch(H))$, there is a monochromatic copy of some graph $H'$ in $G'$, where $H' \in \ch(H)$. If $H'$ is blue, then there is a monochromatic $H'$ with color $1$ in $G$, which contains a $H$, as desired. On the other hand, if $H'$ is red, then since $H'$ is connected, there is a monochromatic copy of $H'$ in color $i$ for some $i \in \{2, 3, \ldots, k\}$ in $G$, which contains a $H$, as desired.
\end{proof}

Lemma~\ref{Lem:CH} provides a general framework for proving that $gr_k(P_5 : H)=R_3(H)$ for various graphs $H$. For example, we will use Lemma~\ref{Lem:CH} to prove that $gr_k(P_5 : mK_3)=R_3(mK_3)$ and $gr_k(P_5 : mC_5)=R_3(mC_5)$ and others in Subsection~\ref{Subsec:CH}.

\subsection{Applications of Lemma~\ref{Lem:CH}}\labelz{Subsec:CH}

In order to apply Lemma~\ref{Lem:CH}, we must compute (or at least bound) $R_{2}(\ch(H))$. We therefore state the following propositions which compute this value for triangles and $5$-cycles.

\begin{proposition}[\cite{MR3539092}]\labelz{Prop:mK3} 
For $m\geq 2$, $R_2(\ch(mK_3))=7m-2$.
\end{proposition}

\begin{proposition}\labelz{Prop:mC5} 
For $m\geq 2$, $R_2(\ch(mC_5))=11m-2$.
\end{proposition}

We will provide the proof of Proposition~\ref{Prop:mC5} later, but first we apply Propositions~\ref{Prop:mK3} and~\ref{Prop:mC5} in the following result.

\begin{corollary}\labelz{Cor:mK3mC5} 
For $m\geq 2$,
\bd
\item{\rm (1) } $gr_k(P_5 : mK_3)=R_3(mK_3)$;
\item{\rm (2) } $gr_k(P_5 : mC_5)=R_3(mC_5)$.
\ed
\end{corollary}

\begin{proof}
In order to show that $R_3(H)\geq R_2(\ch(H))$ for $H \in \{mK_{3}, mC_{5}\}$, it suffices to construct a $3$-coloring of a complete graph of order at least $R_{2}(\ch(H)) - 1$ which contains no monochromatic copy of $H$.

First suppose $H = mK_{3}$. For $i \in \{1, 2\}$, let $G_{i}$ be a complete graph of order $3m - 1$ colored entirely with color $i$. Let $G_{3}$ be a complete graph of order $m - 1$ colored entirely with color $3$ and let $G = G_{1} \cup G_{2} \cup G_{3}$ with all edges between these graphs having color $3$. This graph $G$ is a coloring of the complete graph of order $7m - 3$ and contains no monochromatic copy of $mK_{3}$. See Figure~\ref{Fig:K3C5}(a).

Next suppose $H = mC_{5}$. For $i \in \{1, 2\}$, let $G_{i}$ be a complete graph of order $5m - 1$ colored entirely with color $i$. Let $G_{3}$ be a complete graph of order $m - 1$ colored entirely with color $3$ and let $G = G_{1} \cup G_{2} \cup G_{3}$ with all edges between these graphs having color $3$. This graph $G$ is a coloring of the complete graph of order $11m - 3$ and contains no monochromatic copy of $mC_{5}$. See Figure~\ref{Fig:K3C5}(b).

Finally, since $R_3(H)\geq R_2(c(H))$ for $H \in \{mK_{3}, mC_{5}\}$, by Lemma~\ref{Lem:CH}, we have the desired result.
\end{proof}


\begin{figure}[H]
\begin{center}
\includegraphics{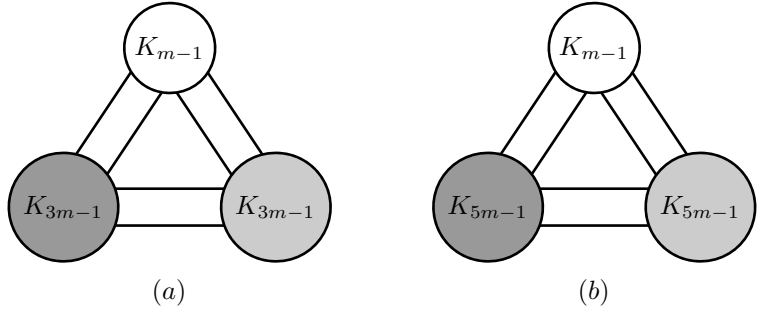}
\caption{Colored complete graphs with no monochromatic $mK_{3}$ or $mC_{5}$ respectively \labelz{Fig:K3C5}}
\end{center}
\end{figure}

Before the proof of Proposition~\ref{Prop:mC5}, we first provide some supporting lemmas.

\begin{lemma}[\cite{MR635872}]\labelz{le:1.1}
For every connected graph $G$ containing at least one edge, we have $R_2(G)\geq (\chi (G)-1)(|G|-1)+s(G)$, where $s(G)$ is the minimum number of vertices in some color class under all proper vertex colorings in $\chi (G)$ colors.
\end{lemma}

\begin{lemma}[\cite{MR1375096}]\labelz{le:1.2}
$8m-1 \leq R_2(mC_5) \leq 8m+1$.
\end{lemma}

\begin{lemma}\labelz{le:1.3}
For $m\geq n \geq 1$, $R(\ch(mC_5), nK_2)=5m+n-1$.
\end{lemma}

\begin{proof}
For the lower bound, let $G_{1}$ be a complete graph of order $5m - 1$ colored entirely with color $1$ and let $G_{2}$ be a complete graph of order $n - 1$ colored entirely with color $2$. Then let $G = G_{1} \cup G_{2}$ where all edges between the two graphs have color $2$. This graph has order $5m + n - 2$ and contains no copy of $mC_{5}$ (so certainly no copy of a graph in $\ch(mC_{5})$) in color $1$ and no copy of $nK_{2}$ in color $2$. See Figure~\ref{Fig:C5K2}.

We prove the upper bound by induction on $n$. For $n=1$, the result is trivial. Let $G$ be a red-blue coloring of $K_{5m + n - 1}$ and suppose for a contradiction that there is no red copy of a graph in $\ch(mC_{5})$ and no blue copy of $nK_{2}$. By induction on $n$, we may assume there is a blue matching $M=(n-1)K_2$. Since there is no blue $nK_2$, every edge $e_i\in M$ contains a vertex $v_i$ adjacent in red to all but at most one vertex of $X=V(G)\setminus V(M)$. (Note that both ends of $e_{i}$ can have a red edge to a single vertex but then no other red edges.) Additionally, $X$ induces a red complete graph.

Since $|X|=5m+n-1-2(n-1)\geq 4n+1$, we can select $n-1$ pairwise disjoint red copies of $P_4$ within $X$ and $n-1$ corresponding vertices in $M$ (using one end from each matching edge) to form a red copy of $(n-1)C_5$. Since there are $5m+n-1-2(n-1)-4(n-1)=5(m-n+1)$ vertices in the remainder of $X$, we can find a red copy of $(m-n+1)C_5$ on these vertices. It is easy to see that the $(n-1)C_5$ and the $(m-n+1)C_5$ can be included into a connected red subgraph, producing the desired red copy of a graph in $\ch(mC_5)$.
\end{proof}

\begin{figure}[H]
\begin{center}
\includegraphics{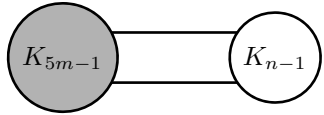}
\caption{Colored complete graph with no monochromatic $mC_{5}$ or $nK_{2}$ \labelz{Fig:C5K2}}
\end{center}
\end{figure}

\noindent {\it Proof of Proposition~\ref{Prop:mC5}.} The lower bound follows from the following example (much like Figure~\ref{Fig:K3C5}(b)). Let $G_{1}$ be copy of $K_{5m - 1}$ colored entirely with color $1$ and let $G_{2}$ be a copy of $K_{m - 1}$ colored entirely with color $2$. The desired example is then the graph $G = 2G_{1} \cup G_{2}$ of order $2(5m - 1) + (m - 1) = 11m - 3$ where all edges between these three subgraphs have color $2$. In $G$, there is no connected subgraph in color $1$ of order at least $5m$ and no copy of $mC_{5}$ in color $2$ so $G$ contains no graph in $\ch(mC_{5})$. Note that this lower bound is not an immediate corollary of Lemma~\ref{le:1.1} since we are considering a set of graphs as opposed to a single graph.

For the upper bound, consider a $2$-coloring $G$ of $K_{11m-2}$, say using red and blue. At least one of these colors must be connected so without loss of generality, suppose blue induces a connected subgraph. By Lemma~\ref{le:1.2} there is a monochromatic copy of $mC_5$. In order to avoid a monochromatic copy of a graph in $\ch(mC_5)$, this copy of $mC_5$ must be red and red must induce a disconnected subgraph. Suppose the red components have vertex sets $X_1, X_2, \ldots, X_s$ where $s\geq 2$ and suppose $|X_1|\geq \cdots \geq |X_s|$ so $\frac{11m - 2}{s} \leq |X_1| \leq 8m$. Note that all edges between these components $X_{i}$ must be blue.

We break the remainder of the proof into cases based on the orders of $X_{1}$ and $X_{2}$.

\setcounter{case}{0}
\begin{case}
$|X_1|\geq 5m$ and $|X_2|\geq 5m$.
\end{case}

Suppose $|X_i|=5m+k_i$ for $i=1, 2$. Note that $k_1+k_2\leq 11m-2-5m-5m=m-2$. By Lemma~\ref{le:1.3}, there is a blue matching $M_i=(k_i+1)K_2$ within $X_i$ for each $i=1, 2$. Since
$$3(k_i+1)+2(k_{3-i}+1)=2(k_i+k_{3-i}+2)+k_{i}+1\leq 2m+k_{i}+1\leq |X_{i}|$$
for $i=1, 2$, we can take $2(k_i+1)$ vertices in $X_{3-i}\setminus M_{3-i}$, $k_i+1$ vertices in $X_i\setminus M_i$ and $V(M_i)$ to form a blue copy of $(k_i+1)C_5$, denoted by $C'_i$. Note that we can choose $C'_{1}$ and $C'_{2}$ to be vertex disjoint. If $k_1+k_2+2=m$, then we are done since $C'_{1} \cup C'_{2}$ along with any combination of the many blue edges between them produces a blue graph in $\ch(mC_{5})$. 

Thus, suppose $\ell=m-k_1-k_2-2 >0$, so there are $(11m-2)-(5m+k_1)-(5m+k_2)=\ell$ vertices in the set $A=V(G)\setminus (X_1\cup X_2)$. Since
\beqs
|X_i\setminus (C'_1\cup C'_2)| &=& 5m+k_i-3(k_i+1)-2(k_{3-i}+1) \\
~ & = & 5m-2(k_1+k_2)-5\\
~ & = & 3m-1+2(m-k_1-k_2-2)\\
~ & \geq & 2\ell
\eeqs
for $i=1, 2$, we can choose $\ell$ vertices in $A$, $2\ell$ vertices in $X_1\setminus (C'_1\cup C'_2)$, and $2\ell$ vertices in $X_2\setminus (C'_1\cup C'_2)$ to form a blue copy of $\ell C_5$ that is disjoint from $C'_{1}$ and $C'_{2}$. Then we get a blue $mC_5$, and since the blue subgraph is connected, we have produced the desired blue graph in $\ch(mC_5)$.

\begin{case}
$|X_1|\geq 5m$ and $3m\leq |X_2|< 5m$.
\end{case}

Suppose without loss of generality that $|X_1|=5m+k_1$ and $|X_2|=5m-k_2$, where $1\leq k_2 \leq 2m$. If $k_1\geq m-1$, then there is blue copy of $M_1=mK_2$ within $X_1$ by Lemma \ref{le:1.3}. We can then find a blue copy of $mC_5$ by taking $2m$ vertices in $X_2$, $m$ vertices in $X_1\setminus M_1$, and $V(M_1)$, thereby creating the desired blue graph in $\ch(mC_5)$. Therefore, we may assume that $k_1\leq m-2$. By Lemma~\ref{le:1.3}, there is a blue copy of $M'_1=(k_1+1)K_2$ within $X_1$, and so we can find a blue copy of $(k_1+1)C_5$, denoted by $C'$, by using $2(k_1+1)$ vertices in $X_2$, $k_1+1$ vertices in $X_1\setminus M'_1$, and $V(M'_1)$.

If $1\leq k_2 \leq m$, then $A=V(G)\setminus (X_1\cup X_2)$ contains at least $11m-2-5m-k_1-5m+k_2=m-k_1+k_2-2$ vertices, $|X_1\setminus C'|=5m+k_1-3(k_1+1)\geq 2(m-k_1+k_2-2)$, and $|X_2\setminus C'|=5m-k_2-2(k_1+1)\geq 2(m-k_1+k_2-2)$. Thus, we can choose $m-k_1+k_2-2$ vertices in $A$, $2(m-k_1+k_2-2)$ vertices in $X_1\setminus C'$, and $2(m-k_1+k_2-2)$ vertices in $X_2\setminus C'$ to form a blue $(m-k_1+k_2-2)C_5$. Since $k_1+1+m-k_1+k_2-2\geq m$, we obtain a blue $mC_5$ and so we have the desired blue graph in $\ch(mC_5)$.

If $m< k_2 \leq 2m$, then $A=V(G)\setminus (X_1\cup X_2)$ contains $|A|=11m-2-5m-k_1-5m+k_2\geq m+1$ vertices. Then we can find a blue $mC_5$ by taking $2m$ vertices in $X_1$, $2m$ vertices in $X_2$, and $m$ vertices in $A$. This produces the desired blue graph in $c(mC_5)$.

\begin{case}
$|X_1|\geq 5m$ and $|X_2|< 3m$, or $|X_1|<5m$.
\end{case}

In this case, since $|V(G)\setminus X_1|\geq 11m-2-8m \geq 2m$, we may assume that $|X_1|\leq 6m-2$ since otherwise we can find a blue copy of $mC_5$ similarly as Case~2.

Let $A$ be a set of vertices such that $|A|=2m$ and, starting with $X_1$, all vertices of $X_i$ are selected before taking vertices from $X_{i+1}$. Then starting at the next set $X_i$, we choose $B$ in the same way such that $|B|=2m$. Let $C=\{\cup X_i: (A\cup B)\cap X_i=\emptyset\}$. It is easy to see that $|C|\geq m$. Thus we can find a blue $mC_5$ by taking all of $A$, all of $B$, and $m$ vertices from $C$. This produces the desired blue graph in $\ch(mC_5)$.
\hfill \qed 




The method in the proof of Proposition~\ref{Prop:mC5} can be used to consider Case (b) in rainbow $P_5$-free colorings, since the structures are analogous. It is also similar to the strategy when considering bipartite graphs in Lemma~\ref{Lem:Bip}.

\begin{theorem}\labelz{Thm:AlmostBip} Let $G=G(S,T)$ be a bipartite graph, and let $H$ be the graph
obtained from $G$ by adding an edge within $S$. For any integer $m\geq 2$, we have $gr_k(P_5 : mH)=R_3(mH).$
\end{theorem}

\begin{proof}
Suppose $k\geq 4$ and let $G$ be a rainbow $P_5$-free $k$-coloring of $K_{R_3(mH)}$. By Lemma~\ref{Lem:P5Disconnected+b}, we may assume $G$ satisfies Case (b) of Theorem~\ref{Thm:P5Classify}. Then $V(G)$ has a partition $V_2, V_3, \ldots, V_k$, where $|V_2|\geq |V_3|\geq \cdots \geq |V_k|\geq 2$.

Let $|S|=s$, $|T|=t$, and we may assume that $s+t\geq 3$. First, we have the following results concerning Ramsey numbers.

\begin{claim}\label{cl:1b}
$R_3(mH)\geq 3m(s+t)-2$.
\end{claim}

\begin{proof}
Let $U_1, U_2, U_3$ be three vertex sets with $|U_1|=|U_2|=|U_3|=m(s+t)-1$.
For $i=1, 2, 3$, we color the edges within $U_i$ with color $i$, and color the edges between $U_1, U_2, U_3$
such that $c(U_1, U_2)=3$, $c(U_2, U_3)=1$ and $c(U_3, U_1)=2$. The resulting coloring is a 3-coloring
of $K_{3m(s+t)-3}$ without monochromatic $mH$.
\end{proof}

\begin{claim}\labelz{cl:2}
For $m\geq n \geq 1$, $R_2(mH, nK_2)=m(s+t)+n-1$.
\end{claim}

\begin{proof}
For the lower bound, let $U_1, U_2$ be two vertex sets with $|U_1|=m(s+t)-1$ and $|U_2|=n-1$. We use red to color the edges within $U_1$ and use blue to color all the remaining edges. The resulting coloring of $K_{m(s+t)+n-2}$ contains no red $mH$ and no blue $nK_2$.

For the upper bound, we will prove by induction on $n$. For $n=1$, it is trivial. Suppose we have a blue matching $M=(n-1)K_2$ in a 2-coloring $\Gamma$ of $K_{m(s+t)+n-1}$ with red and blue. If there is no blue $nK_2$, then every edge $e_i=u_iv_i$ in $M$ contains a vertex, say $u_i$, adjacent in red to all but at most one vertex of $X=V(\Gamma)\setminus V(M)$, and $X$ induces a red complete graph. Since $|X|=m(s+t)+n-1-2(n-1)\geq (n-1)(s+t-1)+s+t$, we can find $n-1$ pairwise disjoint red $H$ using $\{u_1, u_2, \ldots, u_{n-1}\}$ and $(n-1)(s+t-1)$ vertices in $X$. Since there are $m(s+t)+n-2-2(n-1)-(n-1)(s+t-1)=(m-n+1)(s+t)$ vertices in the remainder of $X$, we can find a red $(m-n+1)H$. The result follows.
\end{proof}

Note that since $|G|=R_3(mH)$, it is easy to see that $|V_3\cup V_4\cup \cdots \cup V_k|\geq |mH|=m(s+t)$. Now we can give an upper bound of $|V_2|$.

\begin{claim}\labelz{cl:3}
$|V_2|\leq m(s+t)+m-2$.
\end{claim}

\begin{proof}
If $|V_2|\geq m(s+t)+m-1$, then by Claim~\ref{cl:2} there is a matching $M=mK_2$ in color 1 within $V_2$ for avoiding a $mH$ in color 2. Hence, we can form a $mH$ in color 1 by taking $V(M)$, $m(s-2)$ vertices in $V_2\setminus V(M)$ and $mt$ vertices in $V_3\cup \cdots \cup V_k$, a contradiction.
\end{proof}

By Claims~\ref{cl:1b} and~\ref{cl:3}, we have $2\leq |V_k|\leq \cdots \leq |V_2|\leq m(s+t)+m-2$ and $|V_4\cup \cdots \cup V_k|\geq 3m(s+t)-2-2(m(s+t)+m-2)=m(s+t)-2m+2\geq m+2$. In the following, we will divide the rest of the proof into three cases.

\setcounter{case}{0}
\begin{case}
$|V_2|\geq m(s+t)$ and $|V_3|\geq mt$.
\end{case}

Since $|V_4\cup \cdots \cup V_k|\geq m+2$, we can form a $mH$ in color 1 by taking $m(s-1)$ vertices in $V_2$, $mt$ vertices in $V_3$ and $m$ vertices in $V_4\cup \cdots \cup V_k$, a contradiction.

\begin{case}
$|V_2|\geq m(s+t)$ and $|V_3|\leq mt-1$.
\end{case}

Firstly, we choose $A\subseteq V_2$ with $|A|=m(s-1)$. Secondly, we choose $B\in V_3\cup V_4\cup \cdots \cup V_k$ with $|B|=mt$ starting with $V_3$, and all the vertices of $V_i$ are selected before taking vertices from $V_{i+1}$ for $i\geq 3$. Suppose for some $j>3$ we have $B\cap V_j\neq \emptyset$ and $B\cap V_{j+1}=\emptyset$, i.e., $|V_3\cup \cdots \cup V_{j-1}|\leq mt-1$ and $|V_3\cup \cdots \cup V_{j}|\geq mt$. Since $|V_j|\leq |V_3|\leq mt-1$, we have $|V_3\cup \cdots \cup V_{j}|=|V_3\cup \cdots \cup V_{j-1}|+|V_j|\leq 2(mt-1)$. Thus, by Claims \ref{cl:1b} and \ref{cl:3} we have
\begin{align*}
  |V_{j+1}\cup \cdots \cup V_{k}| \geq \ & 3m(s+t)-2-(m(s+t)+m-2)-2(mt-1) \\
    = \ & 2ms-m+2\geq m.
\end{align*}
We can choose $C\subseteq V_{j+1}\cup \cdots \cup V_{k}$ such that $|C|=m$. Then $A$, $B$ and $C$ form a $mH$ in color 1, a contradiction.

\begin{case}
$|V_2|\leq m(s+t)-1$.
\end{case}

In this case, we have $|V_4\cup \cdots \cup V_k|\geq 3m(s+t)-2-2(m(s+t)-1)=m(s+t)$. Let $x=\mbox{max}\{s-1,t\}$ and $y=\mbox{min}\{s-1,t\}$. We shall show that $G$ contains a $mK_{1,s-1,t}$ in color 1.

If $|V_2|\geq mx$, then we choose $A\subseteq V_2$ with $|A|=mx$ and let $l=3$. If $|V_2|\leq mx-1$, then we choose $A\in V_2\cup V_3\cup \cdots \cup V_k$ with $|A|=mx$ starting with $V_2$, and all the vertices of $V_i$ are selected before taking vertices from $V_{i+1}$ for $i\geq 2$. Suppose for some $j_1>2$ we have $A\cap V_{j_1}\neq \emptyset$ and $A\cap V_{j_1+1}=\emptyset$, i.e., $|V_2\cup \cdots \cup V_{j_1-1}|\leq mx-1$ and $|V_2\cup \cdots \cup V_{j_1}|\geq mx$. Since $|V_{j_1}|\leq |V_2|\leq mx-1$, we have $|V_2\cup \cdots \cup V_{j_1}|=|V_2\cup \cdots \cup V_{j_1-1}|+|V_{j_1}|\leq 2(mx-1)$. Now we let $l=j_1+1$.

If $|V_l|\geq my$, then we choose $B\subseteq V_l$ with $|B|=my$ and let $l'=l+1$. If $|V_l|\leq my-1$, then we choose $B\in V_l\cup \cdots \cup V_k$ with $|B|=my$ starting with $V_l$, and all the vertices of $V_i$ are selected before taking vertices from $V_{i+1}$ for $i\geq l$. Suppose for some $j_2>l$ we have $B\cap V_{j_2}\neq \emptyset$ and $B\cap V_{j_2+1}=\emptyset$, i.e., $|V_l\cup \cdots \cup V_{j_2-1}|\leq my-1$ and $|V_l\cup \cdots \cup V_{j_2}|\geq my$. Since $|V_{j_2}|\leq |V_l|\leq my-1$, we have $|V_l\cup \cdots \cup V_{j_2}|=|V_l\cup \cdots \cup V_{j_2-1}|+|V_{j_2}|\leq 2(my-1)$. Now we let $l'=j_2+1$.

Note that we have $|V_2\cup \cdots \cup V_{l'-1}|\leq \mbox{max}\{2(m(s+t)-1), m(s+t)-1+2(mx-1)\}$. Thus $|V_{l'}\cup \cdots \cup V_{k}|\leq 3m(s+t)-2-\mbox{max}\{2(m(s+t)-1), m(s+t)-1+2(mx-1)\}\geq m$. We can choose $C\subseteq V_{l'}\cup \cdots \cup V_{k}$ such that $|C|=m$. Then $A$, $B$ and $C$ form a $mK_{1,s-1,t}$ in color 1, which contains a $mH$, a contradiction.
\end{proof}

Next some results about complete graphs.

\begin{theorem}\labelz{Thm:nKr} 
For $r \geq 3$ and $2\leq n\leq r-1$, we have $nK_r \notin \mathscr{H}$.
\end{theorem}

\begin{proof}
First a claim about the $2$-color Ramsey number for $\ch(nK_{r})$.

\begin{claim}\labelz{cl:1.1}
$R_2(\ch(nK_r))\leq (r-1)(R_2(K_r)-1)+n$.
\end{claim}

\begin{proof}
In the case $r=3$, the result follows from \cite{MR3539092} where it was shown that $R_2(\ch(mK_3))=7m-2$ for $m\geq 2$. Thus, we may assume that $r\geq 4$. For a contradiction, suppose that $G$ is a $2$-coloring of a complete graph $K_N$ using red and blue without monochromatic $\ch(nK_r)$, where $N=(r-1)(R_2(K_r)-1)+n=(r-1)R_2(K_r)-r+1+n$. Without loss of generality, we may assume that red is connected. Then there are at most $n-1$ disjoint red copies of $K_r$, and denote these by $R_1, R_2, \ldots, R_{t}$, where $t\leq n-1$. Let $G'=G \setminus (R_1 \cup R_2 \cup \dots \cup R_t)$, and note that there is no red copy of $K_r$ in $G'$. Moreover, we have
\beqs
 |G'| & \geq & (r-1)R_2(K_r)-r+1+n-(n-1)r \\
 &  = & (r-2)R_2(K_r)+R_2(K_r)-nr+n+1.
\eeqs

\setcounter{case}{0}
\begin{case}
$r\geq 6$.
\end{case}

From the known lower bounds of Ramsey numbers for small complete graphs, we have $R_2(K_r)\geq 2r^2$ for $6\leq r\leq 14$, and from \cite{MR0019911}, we have $R_2(K_r)> \frac{r}{e\sqrt{2}}2^{r/2}\geq 2r^2$ for $r\geq 15$. Thus, for all $r\geq 6$ we have $R_2(K_r)\geq 2r^2$. Therefore, we have
\beqs
|G'| & \geq & (r-2)R_2(K_r)+2r^2-nr+n+1 \\
~ & \geq & (r-2)R_2(K_r) \\
~ & \geq & R_2(K_r) +2r^2(r-3) \\
~ & \geq & R_2(K_r).
\eeqs
Since there is no red copy of $K_{r}$ in $G'$, there must be a blue copy of $K_r$ in $G'$. In fact, we can find $2r(r-3)$ disjoint copies of blue $K_r$ in $G'$ by a simple greedy application of the same argument. Call these blue cliques $B_1, B_2, \ldots, B_{2r(r-3)}$. Consider the $2$-coloring $S$ of $K_{2r(r-3)}$ by taking one vertex from each clique $B_i$ for $i \in \{1, 2, \ldots, 2r(r-3)\}$. Since 
$$
R(K_r, T_n)=(r-1)(n-1)+1 \leq 2r(r - 3)
$$
for any tree $T_n$ of order $n$ (see \cite{MR0321783}), there is a blue copy of $T_n$ in $S$. This gives a blue copy of a graph in $\ch(nK_r)$ in $G$, a contradiction.

\begin{case}
$r = 5$.
\end{case}

In this case, we have $R_2(K_5)\geq 43\geq 8r$. Thus
\beqs
|G'| & \geq & (r-2)R_2(K_r)+R_2(K_r)-nr+n+1 \\
~ & = & (r-2)R_2(K_r)+(8-n)r+n+1 \\
~ & \geq & R_2(K_r) +(r-3)R_2(K_r) \\
~ & \geq & R_2(K_r)+8r(r-3) \\
~ & = & R_2(K_r)+16r.
\eeqs
Therefore, we can greedily find $16$ vertex disjoint blue copies of blue $K_5$ in $G'$. Since $R(K_5, T_4)=13<16$, we can find the desired blue copy of a graph in $\ch(nK_5)$ by a same argument as above.

\begin{case}
$r= 4$.
\end{case}

In this case, we have $R_2(K_4)=18$. Thus we have
\beqs
|G'| & \geq & (r-2)R_2(K_r)+R_2(K_r)-nr+n+1 \\
~ & = &(r-3)R_2(K_r)+2R_2(K_r)-n(r-1)+1 \\
~ & \geq & R_2(K_r)+2\cdot 18-9+1 \\
~ & = & R_2(K_r)+28.
\eeqs
Therefore, we can greedily find $7$ vertex disjoint blue copies of $K_5$ in $G'$. Since $R(K_4, T_3)=7$, we can find the desired blue copy of a graph in $\ch(nK_4)$ by a same argument as above.
\end{proof}

In the remainder of the proof, we show that $R_3(nK_r)\geq (r-1)(R_2(K_r)-1)+n\geq R_2(\ch(nK_r))$ by constructing a 3-coloring of $K_{N'}$ where $N'=(r-1)(R_2(K_r)-1)+n-1$, which contains no monochromatic copy of $nK_{r}$. Let $T$ be a $2$-coloring of $K_{R_{2}(K_r)-1}$ without monochromatic $K_r$ using red and blue. Let $T_1, T_2, \ldots, T_{r-1}$ be $r-1$ disjoint copies of $T$. Let $T_r$ be a monochromatic $K_{n-1}$ with green. We form a $K_{N'}$ by adding only green edges between $T_1, T_2, \ldots, T_r$. Since each green $K_r$ must contains a vertex of $T_r$, there is no green copy of $nK_r$, and clearly there is no red or blue copy of $nK_r$. Thus we have $R_3(nK_r)\geq R_2(\ch(nK_r))$, and so $gr_k(P_5 : nK_r)=R_3(nK_r)$.
\end{proof}

\begin{theorem}\labelz{th:2} 
For any connected graph $G$, we have $2G \notin \mathscr{H}$.
\end{theorem}

\begin{proof}
For a contradiction, suppose $\Gamma$ is a rainbow $P_5$-free $k$-coloring
of $K_{R_3(2G)}$ without a monochromatic copy of $2G$. If $\chi(G)\leq 2$, then $2G$ is a bipartite graph, and the result is true by Lemma~\ref{Lem:Bip}.

If $\chi(G)=3$, then $G=G(X, Y, Z)$ is a tripartite graph. By Lemma~\ref{Lem:P5Disconnected+b}, we may assume that $\Gamma$ satisfies Case~(b) of Theorem~\ref{Thm:P5Classify}. Let $V_2, V_3, \ldots, V_k$ be a partition of $V(\Gamma)$ such that there are only colors $1$ or $i$ within $V_i$ for $i\in \{2, 3, \ldots, k\}$ and only color $1$ on edges between these parts. Without loss of generality, suppose that $|V_2|=\max \{|V_2|, |V_3|, \ldots, |V_k|\}$.
Now we recolor the edges of $\Gamma$ to make a $3$-coloring such that
\bd
\item{\rm (1)} all edges of color $1$ become red;
\item{\rm (2)} all edges of color $2$ become blue; and
\item{\rm (3)} all edges of other colors become green.
\ed

Let $\Gamma'$ denote the resulting $3$-coloring. Since $\Gamma$ contains no monochromatic $2G$ and $|\Gamma'|=|\Gamma|=R_3(2G)$, there is a monochromatic copy of $2G$ in $\Gamma'$ and to avoid such a subgraph in $\Gamma$, this subgraph must be green. Moreover, these two green copies of $G$ must appear in distinct parts, say $V_3$ and $V_4$. Without loss of generality, suppose $|X|\geq |Y|\geq |Z|$. Since $|V_2|\geq |V_3|\geq 3|Z|$, we can construct two disjoint copies of $G$ in color $1$ in $\Gamma$, one by taking $|Z|$ vertices in $V_2$, $|X|$ vertices in $V_3$, and $|Y|$ vertices in $V_4$, and the other one by taking $|Z|$ vertices in $V_2$, $|Y|$ vertices in $V_3$, and $|X|$ vertices in $V_4$, a contradiction.

We may therefore assume $\chi(G)\geq 4$. We have the following claim.

\noindent\begin{claim}\labelz{cl:2.1}
If $\chi (G)\geq 4$, then $R_2(\ch(2G)) \leq (\chi (G)-1)(R_2(G)-1)+2s(G)$, where $s(G)$ is the minimum number of vertices in some color class over all proper vertex-colorings in $\chi (G)$ colors.
\end{claim}

\begin{proof}
For a contradiction, suppose $H$ is a 2-coloring of $K_N$ using red and blue without monochromatic copy of a graph in $\ch(2G)$, where $N=(\chi (G)-1)(R_2(G)-1)+2s(G)=(\chi (G)-1)R_2(G)-\chi(G)+1+2s(G)$. Without loss of generality, we may assume that red is connected. Then there is at most one red copy of $G$, denoted by $R$. Let $H'=H-R$, and so there is no red copy of $G$ in $H'$. Moreover, we have
\beqs
|H'| & \geq & (\chi (G)-1)R_2(G)-\chi(G)+1+2s(G)-|G| \\
~ & = & (\chi (G)-3)R_2(G)+2R_2(G)-\chi(G)+1+2s(G)-|G| \\
~ & \geq & (\chi (G)-3)R_2(G)+2[(\chi(G)-1)(|G|-1)+s(G)]-\chi(G)+1\\
~ & ~ & + 2s(G)-|G| \\
~ & \geq & (\chi (G)-3)R_2(G)+6(|G|-1)-\chi(G)-|G|+4s(G)+1 \\
~ & \geq & (\chi (G)-3)R_2(G)+4|G|+4s(G)-5 \\
~ & \geq & (\chi (G)-3)R_2(G)+3|G|.
\eeqs
Thus we can greedily find $\chi(G)$ disjoint blue copies of $G$ in $H'$, say with vertex sets $B_1, B_2, \ldots, B_{\chi(G)}$. In order to avoid a blue copy of a graph in $\ch(2G)$, there must be only red edges in between the sets of vertices $B_1, B_2, \ldots, B_{\chi(G)}$.

Consider a proper vertex-coloring of $G$ with $\chi(G)$ colors, and let $c_i(G)$ be the number of vertices colored by color $i$ for $i\in \{1, 2, \ldots, \chi(G)\}$. We first form a red copy of $G$ in $H$, denoted by $R_1$, by taking $c_i(G)$ vertices in $V(B_i)$ for each $i\in \{1, 2, \ldots, \chi(G)\}$. Then we form a second red copy of $G$ in $H\setminus R_1$, denoted by $R_2$, by selecting $c_{i+1}(G)$ vertices in $V(B_{i}\setminus R_1)$ for $i\in \{1, 2, \ldots, \chi(G)\}$ (where $c_{\chi(G)+1}=c_1$). Therefore, there is a red copy of a graph in $\ch(2G)$ within $H$, a contradiction.
\end{proof}

In the remainder of the proof, we will show that $R_3(2G)\geq (\chi (G)-1)(R_2(G)-1)+2s(G)\geq R_2(\ch(2G))$ by constructing a $3$-coloring of $K_{N'}$, where $N'=(\chi (G)-1)(R_2(G)-1)+2s(G)-1$. Let $T$ be a $2$-coloring of $K_{R_{2}(G)-1}$ without a monochromatic copy of $G$, say using colors red and blue. Let $T_1, T_2, \ldots, T_{\chi(G)-1}$ be $\chi(G)-1$ disjoint copies of $T$. Let $T_{\chi(G)}$ be a monochromatic $K_{2s(G)-1}$ in green. We form a $K_{N'}$ by adding all green edges between the disjoint graphs $T_1, T_2, \ldots, T_{\chi(G)}$. Since each green copy of $G$ must contains at least $s(G)$ vertices from $T_{\chi(G)}$, there can be no green copy of $2G$, and clearly there is no red or blue copy of $2G$ since there is no copy of $G$ in either color. Thus, we have $R_3(2G)\geq R_2(\ch(2G))$, so $gr_k(P_5 : 2G)=R_3(2G)$, as claimed.
\end{proof}

\section{Conclusion}

In light of Lemma~\ref{Lem:P5Disconnected+b}, since almost all graphs are connected, we get the following immediate corollary.

\begin{corollary}
For almost all graphs $H$, we have $gr_{k}(P_{5} : H) = R_{3}(H)$.
\end{corollary}

In order to confirm Conjecture~\ref{Conj:P5Reduction} for more general classes of graphs, we require bounds on the corresponding $2$- and $3$-color Ramsey numbers. As the area of Ramsey Theory develops, more results relating to Conjecture~\ref{Conj:P5Reduction} are likely to become feasible.

Through Lemma~\ref{Lem:CH}, we initiate the discussion of Ramsey numbers of classes of graphs obtained from disconnected graphs by adding edges. While largely unexplored, this area appears to be very fertile for future research.

\begin{problem}
Given a positive integer $k$ and a disconnected graph $H$, find $R_{k}(\ch(H))$.
\end{problem}


\bibliography{../GR-Ref.bib}
\bibliographystyle{plain}

\end{document}